\documentclass[12pt]{article}

\usepackage{graphicx}
\usepackage[a4paper, hmargin={2.5cm,2.5cm},vmargin={3.2cm,3.2cm}]{geometry}
\usepackage[font=small,labelfont=bf]{caption}
\usepackage{subcaption}

\usepackage[T1]{fontenc}
\usepackage[utf8]{inputenc}
\usepackage[english]{babel}

\usepackage{amsmath,amssymb,amsfonts}
\usepackage{amsthm,latexsym}
\usepackage{etoolbox}
\usepackage{extarrows}

\usepackage{cite}
\usepackage{paralist}
\usepackage{tikz}
\usetikzlibrary{shapes}
\usepackage{pgfplots}
\pgfplotsset{compat=newest}
\usepackage{twoopt}

\usepackage{url}
\usepackage[bf,compact,pagestyles,small]{titlesec}%topmarks,calcwidth
\titlespacing*{\section}{0pt}{14pt}{4pt}
\titlespacing*{\subsection}{0pt}{8pt}{3pt}

\usepackage{fancyhdr,lastpage}
\def\maketimestamp{\count255=\time
\divide\count255 by 60\relax
\edef\thetime{\the\count255:}%
\multiply\count255 by-60\relax
\advance\count255 by\time
\edef\thetime{\thetime\ifnum\count255<10 0\fi\the\count255}
\edef\thedate{\number\day-\ifcase\month\or Jan\or Feb\or Mar\or
             Apr\or May\or Jun\or Jul\or Aug\or Sep\or Oct\or
             Nov\or Dec\fi-\number\year}
\def\timstamp{\hbox to\hsize{\tt\hfil\thedate\hfil\thetime\hfil}}}
\maketimestamp
\fancypagestyle{plain}{%
\fancyhf{} % clear all header and footer fields
\fancyfoot[C]{\small \thepage{} of \pageref{LastPage}}}

\numberwithin{equation}{section}  % OR chapter
\allowdisplaybreaks[4]
\newtheorem{thm}{Theorem} %[section]
\newtheorem{lem}[thm]{Lemma}
\newtheorem{prop}[thm]{Proposition}
\newtheorem{cor}[thm]{Corollary}
\theoremstyle{definition}
\newtheorem{conj}{Conjecture} %[section]
\newtheorem*{conj*}{Conjecture} %[section]

\AtEndEnvironment{ex}{\null\hfill\ensuremath{\blacksquare}}% \blacksquare \ensuremath{\blacktriangle}}%
\theoremstyle{remark}
\newtheorem{rem}{Remark}

\DeclareMathOperator{\supp}{supp} %

\DeclareMathOperator{\exponential}{e}

\newcommandtwoopt{\gaborG}[3][a][b]{\mathcal{G}(#3,#1,#2)} %{\{\pi(\nu)#2\}_{\nu \in 
\newcommand{\sfrac}[1]{F{(#1)}}
\newcommand{\round}[1]{R{(#1)}}

\newcommand{\myexp}[1]{\exponential^{#1}}

\newcommand*{\numbersys}[1]{\ensuremath{\mathbb{#1}}}
\newcommand*{\C}{\numbersys{C}}
\newcommand*{\R}{\numbersys{R}}

\newcommand*{\Z}{\numbersys{Z}}

\newcommand*{\N}{\numbersys{N}}

\newcommand*{\cF}{\mathcal{F}}

\newcommand*{\cG}{\mathcal{G}}

\newcommand{\itvoc}[2]{\ensuremath{\left({#1},{#2}\right]}} % InTerValOpenClosed
\newcommand{\itvcc}[2]{\ensuremath{\left[{#1},{#2}\right]}} %
\newcommand{\itvco}[2]{\ensuremath{\left[{#1},{#2}\right)}} %
\newcommand{\itvcos}[2]{\ensuremath{\lbrack{#1},{#2})}} %
\newcommand{\abs}[1]{\ensuremath{\left\lvert#1\right\rvert}}

\newcommand{\set}[1]{\ensuremath{\left\lbrace{#1}\right\rbrace}}

\newcommand{\setprop}[2]{\ensuremath{\left\lbrace{#1} : {#2}\right\rbrace}}

\newcommand{\floor}[1]{\left\lfloor #1 \right\rfloor}

\usepackage{hyperref}
\hypersetup{
 pdfview={FitH},
 pdfstartview={FitH},
 pdfauthor = {Jakob Lemvig, Kamilla Haahr Nielsen}
 pdftitle = {Counterexamples to the B-spline conjecture for Gabor
   frames}, %{On the frame set of Gabor systems with B-spline generators},
 pdfkeywords = {B-spline, frame, frame set, Gabor system},
 pdfcreator = {LaTeX with hyperref package},
 pdfproducer = PDFlatex}

\makeatletter
\def\blfootnote{\xdef\@thefnmark{}\@footnotetext}
\def\subjclass{\xdef\@thefnmark{}\@footnotetext}
\long\def\symbolfootnote[#1]#2{\begingroup%
\def\thefootnote{\fnsymbol{footnote}}\footnote[#1]{#2}\endgroup}
\if@titlepage
  \renewenvironment{abstract}{%
      \titlepage
      \null\vfil
      \@beginparpenalty\@lowpenalty
      \begin{center}%
        \bfseries \abstractname
        \@endparpenalty\@M
      \end{center}}%
     {\par\vfil\null\endtitlepage}
\else
  \renewenvironment{abstract}{%
      \if@twocolumn
        \section*{\abstractname}%
      \else
        \small
        \list{}{%
          \settowidth{\labelwidth}{\textbf{\abstractname:}}
          \setlength{\leftmargin}{50pt}
          \setlength{\rightmargin}{50pt}
          \setlength{\itemindent}{\labelwidth}
          \addtolength{\itemindent}{\labelsep}
        }
        \item[\textbf{\abstractname:}]

      \fi}
      {\if@twocolumn\else\endlist\fi}
\fi
\makeatother

\begin{document}

\title{Counterexamples to the B-spline conjecture for Gabor frames}

\date{\today}

 \author{Jakob Lemvig\footnote{Technical University of Denmark, Department of Applied Mathematics and Computer Science, Matematiktorvet 303B, 2800 Kgs.\ Lyngby, Denmark, E-mail: \protect\url{jakle@dtu.dk}}\phantom{$\ast$},
Kamilla Haahr Nielsen\footnote{Technical University of Denmark, Department of Applied Mathematics
     and Computer Science, Matematiktorvet 303B, 2800 Kgs.\ Lyngby, Denmark, E-mail:
     \protect\url{kamillahn@villagok.dk}}} 


\maketitle

\thispagestyle{plain}
\begin{abstract} 
  The frame set conjecture for B-splines $B_n$, $n \ge 2$, states that the frame set is the maximal set that avoids the known obstructions. We show that any hyperbola of the form $ab=r$, where $r$ is a rational number smaller than one and $a$ and $b$ denote the sampling and modulation rates, respectively, has infinitely many pieces, located around $b=2,3,\dots$, \emph{not} belonging to the frame set of the $n$th order B-spline. This, in turn, disproves the frame set conjecture for B-splines.  On the other hand, we uncover a new region belonging to the frame set for B-splines $B_n$, $n \ge 2$.
\end{abstract}

\section{Introduction}
\label{sec:introduction}

One of the fundamental problems in Gabor analysis is to determine for
which sampling and modulation rates, the corresponding time-frequency
shifts of a given generator constitute a frame. The so-called
\emph{frame set} of a generator $g \in L^2(\R)$ is the parameter
values $(a,b)\in\mathbb{R}_+^2$ for which the associated Gabor system
$\mathcal{G}(g,a,b):= \set{\myexp{2\pi i bm \cdot}g(\cdot-ak)}_{k,m\in
  \Z}$ is a frame for $L^2(\R)$. We denote the frame set by
$\mathcal{F}(g)$ and refer to \cite{MR1843717,MR1946982} for an
introduction to frames and Gabor analysis.

In this note we are interested in the frame set of Gabor systems
generated by B-splines $B_n$. The B-splines are given inductively as
\begin{equation*}
B_1 = \chi_{\itvcc{-1/2}{1/2}}, \quad \text{and} \quad
  B_{n+1} = B_{n} \ast B_{1}, \quad \text{for } n \in \N. 
\end{equation*}
Dai and Sun~\cite{DaiABC2015} recently gave a complete
characterization of the frame set for the first B-spline $B_1$. For
the higher order B-splines, the picture is a lot less complete. Since,
for $n\ge 2$, $B_n$ belongs to the Feichtinger algebra $M^1(\R)$, the
set $\cF(B_n)$ is open in $\R^2$ and $\cF(B_n) \subset
\setprop{(a,b)\in \mathbb{R}^2_+ }{ ab < 1}$, see e.g.,
\cite{MR3232589}.  Christensen put focus on the problem of
characterizing $\cF(B_n)$ in \cite{ChristensenNew2014}, while
Gr\"o{}chenig, in~\cite{MR3232589}, went one step further conjecturing
that the frame set for B-splines of order $n \geq 2$ is
$\mathcal{F}(B_n)=\setprop{(a,b)\in \mathbb{R}^2_+ }{ ab < 1, \, a<n,
  \, b \neq 2,3,\dots}$.  On the other hand, the set $\cF(B_1)$ has a
very complicated structure. This phenomenon is partly explained by the
fact that $B_1$ is the only B-spline that does not belong to
$M^1(\R)$. Indeed, $\cF(B_1)$ is not open and $(a,b)=(1,1) \in
\cF(B_1)$.

The main messages of this note are that the B-spline conjecture is
false and that the frame set for B-splines of all orders must have a
very complicated structure, sharing several similarities with
$\cF(B_1)$. 

Kloos and St\"ockler \cite{MR3218799} and Christensen, Kim, and
Kim~\cite{ChristensenOnGabor2015} reported positive results to support
of the frame set conjecture for B-splines, adding new parameter values
$(a,b)\in\mathbb{R}_+^2$ to the known parts of $\mathcal{F}(B_n)$;
these new values are illustrated in Figure~\ref{fig:frame-set-B2} for
the case of $B_2$.
\begin{figure}[h!]
        \centering
                 \begin{tikzpicture}[scale=3.5]
    % \shade[top color=blue,bottom color=gray!50] (0,0) parabola
    % (1.5,2.25) |- (0,0); \draw (1.05cm,2pt) node[above]
    % {$\displaystyle\int_0^{3/2} \!\!x^2\mathrm{d}x$};

    \fill[white!20!green] (0,0) rectangle (2,1/2);

    \draw[style=help lines] (0,0) grid (3.1,3.1);
    % [step=0.25cm] (1,0) grid +(1,1);

    \draw[->] (-0.2,0) -- (3.2,0) node[right] {$a$}; \draw[->]
    (0,-0.2) -- (0,3.1) node[above] {$b$};

    \foreach \x/\xtext in {1/\frac{n}{2}, 2/n, 3/\frac{3n}{2}}
    \draw[shift={(\x,0)}] (0pt,0pt) -- (0pt,-2pt) node[below]
    {$\xtext$};

    \foreach \y/\ytext in {0.5/\frac{1}{n}, 1/1, 2/2, 3/3}
    \draw[shift={(0,\y)}] (0pt,0pt) -- (-2pt,0pt) node[left]
    {$\ytext$};

    \draw[domain=2:4,black,smooth,variable=\t,line width=1] plot
    ({\t-2},{2/\t});
    \fill[xscale=1,blue,
    domain=0:2,smooth,variable=\t] (0,1/2) plot ({\t},{2/(\t+2)}) |-
    (0,1/2);    
   \draw[domain=2.333:2.6667,smooth,variable=\t,thick,purple] plot
   ({0.833/\t},\t);
%    \filldraw[pink] (0.33,2.5) circle (.5pt);
   \draw[domain=1.75:2.25,smooth,variable=\t,thick,pink] plot ({0.5/\t},\t);
   \draw[domain=1.833:2.167,smooth,variable=\t,thick,pink] plot ({0.333/\t},\t);
   \draw[domain=1.833:2.167,smooth,variable=\t,thick,pink] plot
   ({0.666/\t},\t);
   \draw[domain=1.875:2.125,smooth,variable=\t,thick,pink] plot
   ({0.25/\t},\t);
   \draw[domain=1.875:2.125,smooth,variable=\t,thick,pink] plot
   ({0.75/\t},\t);
   \draw[domain=1.9:2.1,smooth,variable=\t,thick,pink] plot
   ({0.2/\t},\t);
   \draw[domain=1.9:2.1,smooth,variable=\t,thick,pink] plot
   ({0.4/\t},\t);
   \draw[domain=1.9:2.1,smooth,variable=\t,thick,pink] plot
   ({0.6/\t},\t);
   \draw[domain=1.9:2.1,smooth,variable=\t,thick,pink] plot
   ({0.8/\t},\t);
   \draw[domain=1.9167:2.0833,smooth,variable=\t,thick,pink] plot
   ({0.167/\t},\t);
   \draw[domain=1.9167:2.0833,smooth,variable=\t,thick,pink] plot
   ({0.833/\t},\t);
   \draw[domain=2.75:3.25,smooth,variable=\t,thick,pink] plot ({0.5/\t},\t);
   \draw[domain=2.833:3.167,smooth,variable=\t,thick,pink] plot ({0.333/\t},\t);
   \draw[domain=2.833:3.167,smooth,variable=\t,thick,pink] plot
   ({0.666/\t},\t);
   \draw[domain=2.875:3.125,smooth,variable=\t,thick,pink] plot
   ({0.25/\t},\t);
   \draw[domain=2.875:3.125,smooth,variable=\t,thick,pink] plot
   ({0.75/\t},\t);
   \draw[domain=2.9:3.1,smooth,variable=\t,thick,pink] plot
   ({0.2/\t},\t);
   \draw[domain=2.9:3.1,smooth,variable=\t,thick,pink] plot
   ({0.4/\t},\t);
   \draw[domain=2.9:3.1,smooth,variable=\t,thick,pink] plot
   ({0.6/\t},\t);
   \draw[domain=2.9:3.1,smooth,variable=\t,thick,pink] plot ({0.8/\t},\t);
   \draw[domain=2.9167:3.0833,smooth,variable=\t,thick,pink] plot
   ({0.167/\t},\t);
   \draw[domain=2.9167:3.0833,smooth,variable=\t,thick,pink] plot
   ({0.833/\t},\t);
    \fill[xscale=1,yellow, domain=1:2,smooth,variable=\t] (1,.5) plot ({\t}, {1/\t}) |- (1,.5);
    \draw[xscale=0.5,domain=1:2,smooth,variable=\t,dashed,thick] plot ({\t},{1/\t});
    \fill[xscale=0.5,yellow, domain=1:2,smooth,variable=\t] (1,.5) plot ({\t}, {1/\t}) |- (1,.5);
    \draw[xscale=0.25,domain=1:2,smooth,variable=\t,dashed,thick] plot ({\t},{1/\t});
    \fill[xscale=0.25,yellow, domain=1:2,smooth,variable=\t] (1,.5) plot ({\t}, {1/\t}) |- (1,.5);
    \draw[xscale=0.125,domain=1:2,smooth,variable=\t,dashed,thick] plot ({\t},{1/\t});
    \fill[xscale=0.125,yellow, domain=1:2,smooth,variable=\t] (1,.5) plot ({\t}, {1/\t}) |- (1,.5);
    \draw[xscale=0.0675,domain=1:2,smooth,variable=\t,dashed,thick] plot ({\t},{1/\t});
    \fill[xscale=0.0675,yellow, domain=1:2,smooth,variable=\t] (1,.5) plot ({\t}, {1/\t}) |- (1,.5);
    \draw[xscale=0.03125,domain=1:2,smooth,variable=\t,dashed,thick] plot ({\t},{1/\t});
    \fill[xscale=0.03125,yellow, domain=1:2,smooth,variable=\t] (1,.5) plot ({\t}, {1/\t}) |-
    (1,.5);
    \fill[xscale=1,red, domain=2:3.1,smooth,variable=\t] (2,0) plot ({\t}, {1/\t}) |- (2,0);
    \draw[domain=0.31:3.05,smooth,variable=\t,dashed,line width=1] plot ({\t},{1/\t}); 
    \draw[-,thick] (0,1/2) -- (2,1/2); 
    \draw[-,red,line width=1.5] (0,2) -- (1/2,2); 
    \draw[-,red,line width=1.5] (0,3) -- (1/3,3);
    \draw[-,dashed,line width=1] (2,0) -- (2,1/2); 
    \draw[-,line width=1.5,xscale=0.5,black!40!green] (1,1/2) -- (1,1); 
    \draw[-,line width=1.5,xscale=0.25,black!40!green] (1,1/2) -- (1,1);
    \draw[-,line width=1.5,xscale=0.125,black!40!green] (1,1/2) -- (1,1); 
    \draw[-,line width=1.5,xscale=0.0675,black!40!green] (1,1/2) -- (1,1);
    \draw[-,line width=1.5,xscale=0.03125,black!40!green] (1,1/2) -- (1,1); 
    \draw[-,line width=1.5,black!40!green] (0,1) -- (1,1);
    \draw[-,line width=1.7,black!40!green] (1,1/2) -- (1,1);
    \draw[domain=0.31:3.05,smooth,variable=\t,dashed,line width=1] plot ({\t},{1/\t});
  \end{tikzpicture}
  \caption{A sketch of the frame set for $B_n$, $n=2$, for
    $0<ab<1$. Red, pink and purple indicate $(a,b)$-values, where
    $\mathcal{G}(B_2,a,b)$ is not a frame. All other colors indicate
    the frame property. The green region is the classical ``painless
    expansions'' \cite{MR836025}, the yellow region is the result
    from~\cite{ChristensenOnGabor2015}, while the dark green lines
    follow from \cite{MR1756138} or \cite{MR3218799}. The blue region
    is the result in Theorem~\ref{thm:BsplinesNEW}. The pink hyperbola
    pieces $ab=\frac{p}{q}$ are the counterexamples from
    Theorem~\ref{thm:frame-set-Bn} (only illustrated for $q \le 6$).
    The purple
    hyperbolic curve through $(\tfrac{1}{3},\tfrac{5}{2})$ is the
    counterexamples in Theorem~\ref{thm:NEWnonFrame}.
  %   The pink dot at $(\tfrac{1}{3},\tfrac{5}{2})$ is the
  %   counterexample in Theorem~\ref{thm:NEWnonFrame}, while the purple
  %   hyperbolic curve is part of
  % Conjecture~\ref{con:NEW1}.
}
\label{fig:frame-set-B2}
\end{figure}
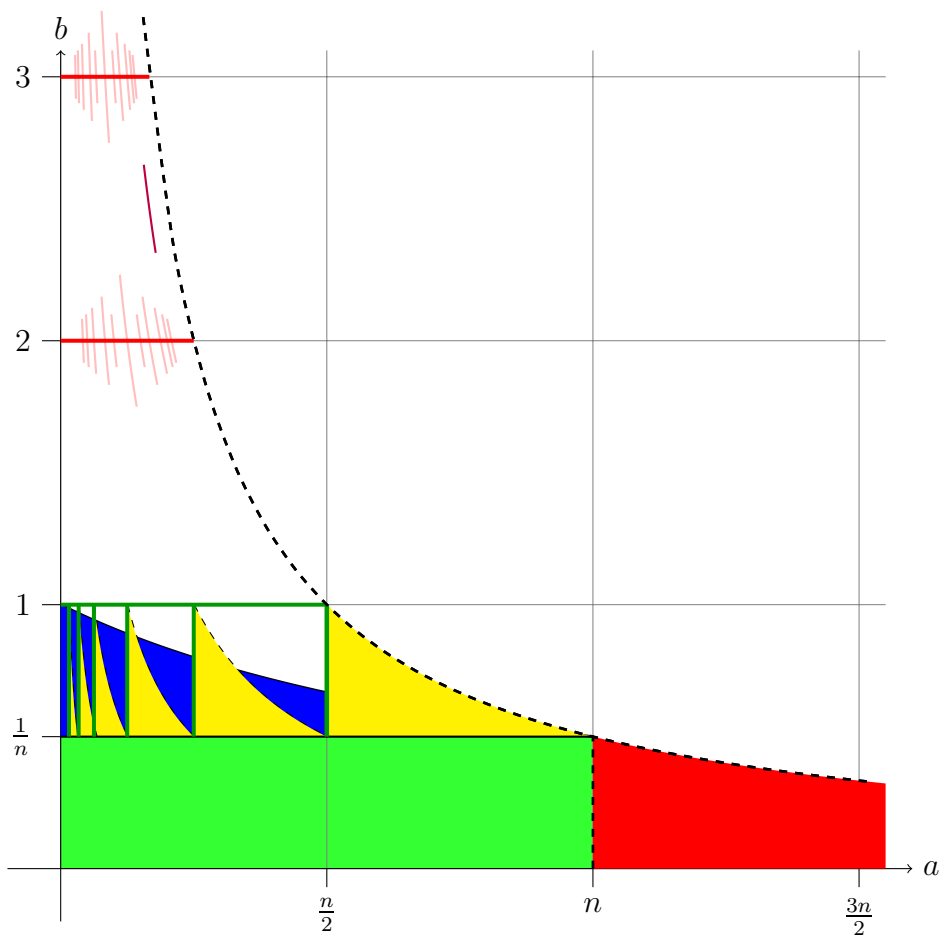
Along the same lines, we verify the frame set conjecture for B-spline
for a new region, marked with blue in Figure~\ref{fig:frame-set-B2}
for the case $B_2$. The idea of the proof is a painless construction
of an alternate dual frame; the details are presented in
Section~\ref{sec:painl-case-altern}.

In Section~\ref{sec:counter-example} we give the proof that the frame
set conjecture is false. Indeed, using a Zibulski-Zeevi representation
and properties of the Zak transform of B-splines, we show in
Theorem~\ref{thm:frame-set-Bn} that any hyperbola of the form
$ab=\frac{p}{q}<1$ with relative prime $p$ and $q$ has infinitely many
pieces not in the frame set $\cF(B_n)$. The pieces are hyperbolic
curves located around $b=2,3,\dots$, with the range of each of these
hyperbolic curve determined by $\abs{b-R(b)} \le \frac{1}{nq}$, where
$\round{x}$ denotes the round function to the nearest integer of $x
\in \R$; a few
of the hyperbolic curves are marked with pink in
Figure~\ref{fig:frame-set-B2} for the case $B_2$. Note that these
pink curves ``color'' a region that leaves a white region (between $b=m$ and
$b=m+1$ for $m=2,3,\dots$) looking like a bow
tie, similar to Janssen's tie for $B_1$ \cite{MR1955931}, but thicker at the knot. We end
Section~\ref{sec:counter-example} by presenting a counterexample for
$B_2$ that is not covered by Theorem~\ref{thm:frame-set-Bn}; in that
subsection we also formulate a new conjecture on the frame set of
$B_2$.

\section{Painless expansions for alternate duals }
\label{sec:painl-case-altern}

For $n>0$, define the parameter set:
\[
  \Sigma_n=\setprop{(a,b)\in \R^2_+}{0< a < n, ab < 1, n + a \leq
    \frac{2}{b}}. 
\]
For $n\in \N, n\ge 2$, we prove that $\mathcal{G}(B_n,a,b)$ is a frame
for $(a,b)\in \Sigma_n$ by constructing a function $h \in L^2(\R)$
such that $\mathcal{G}(B_n,a,b)$ and $\mathcal{G}(h,a,b)$ are dual
frames.  For this, we need the following well-known characterization
result of dual Gabor frames; we refer the reader to \cite{MR1946982}
for a proof.
\begin{thm}\label{thm:dualFrameEqs}
  Let $a,b>0$, and let $g,h\in L^2(\mathbb{R})$. Suppose
  $\mathcal{G}(g,a,b)$ and $\mathcal{G}(h,a,b)$ are Bessel
  sequences. Then they are dual frames if and only if, for all
  $m\in\mathbb{Z}$,
\begin{equation}
\label{eq:dualFrames}
\sum_{k \in \mathbb{Z}} \overline{g(x-m/b-ka)}h(x-ka)=b\,
\delta_{m,0} \quad  \text{for } a.e. \ x\in  \itvcc{-\frac{a}{2}}{\frac{a}{2}}.
\end{equation}
\end{thm}

The following result provides a painless expansion of $L^2$-functions
for a class of compactly supported generators using alternate
duals. The announced frame property of B-splines for $(a,b)$ in the
region $\Sigma_n$ will follow as a special case.
\begin{thm}
\label{thm:painless-alternate}
Let $n>0$, $(a,b) \in \Sigma_n$, and $g \in L^\infty(\R)$.  Suppose
that $\supp{g} \subset \itvcc{-\frac{n}{2}}{\frac{n}{2}}$ and
$c:=\inf_{x \in \itvcc{-\frac{a}{2}}{\frac{a}{2}}} \abs{g(x)}>0$.
Define $h \in L^2(\R)$ by
\[
h(x) = \begin{cases}
\frac{b}{g(x)} & x \in \itvcc{-\frac{a}{2}}{\frac{a}{2}}, \\
0 & \text{otherwise}.
\end{cases}
\]
Then $\mathcal{G}(g,a,b)$ and $\mathcal{G}(h,a,b)$ are dual frames. In
particular, $\mathcal{G}(g,a,b)$ is a frame with lower bound $A_g$
satisfying
\[
 A_g \ge \frac{c^2}{b}.
\]
\end{thm}

\begin{proof}
  It follows by, e.g., \cite[Proposition 6.2.2]{MR1843717}, that
  $\mathcal{G}(g,a,b)$ and $\mathcal{G}(h,a,b)$ are Bessel
  sequences. To show that these systems are dual frames, we will show
  that the equations~(\ref{eq:dualFrames}) hold for a.e.\ $x\in
  \itvcc{-\frac{a}{2}}{\frac{a}{2}}$. For a.e.\ $x\in
  \itvcc{-\frac{a}{2}}{\frac{a}{2}}$, since $\supp{h} =
  \itvcc{-\frac{a}{2}}{\frac{a}{2}}$, we only have one nonzero term in
  the sum in (\ref{eq:dualFrames}). Thus the equations we have to
  verify are
\[
g\bigl(x-\frac{m}{b}\bigr)\, h(x)=b \, \delta_{m,0}, \quad a.e. \ x \in
\itvcc{-\frac{a}{2}}{\frac{a}{2}} \quad \text{for $m \in \mathbb{Z}$}.
\]
 For $m=0$, we have by definition of $h$ that 
\[
g(x)h(x)=g(x) \frac{b}{g(x)}=b, \quad \text{for } a.e. \ x
\in \itvcc{-\frac{a}{2}}{\frac{a}{2}}.
\]
For $m \in \Z \setminus \{0\}$, it follows from our assumption
$\frac{m}{2} + \frac{a}{2} \leq \frac{1}{b}$ that the support sets
\[ 
\supp{g\left(\cdot -\frac{m}{b}\right)}
=\itvcc{-\frac{m}{2}+\frac{m}{b}}{\frac{m}{2}+\frac{m}{b}} \text{ and } \supp{h}
= \itvcc{-\frac{a}{2}}{\frac{a}{2}} 
\]
are disjoint. Hence, for $m \neq 0$, the
equations~\eqref{eq:dualFrames} are trivially satisfied.

 The canonical dual of $\mathcal{G}(h,a,b)$ has optimal lower frame bound
$c^2/b$. Since $\mathcal{G}(g,a,b)$ is an alternate dual of
$\mathcal{G}(h,a,b)$, the bound $A_g \ge c^2/b$ follows.
\end{proof}

\begin{rem}
  In the classical ``painless non-orthogonal expansion'' for compactly
  supported generators~\cite{MR836025}, the Walnut representation of
  the frame operator $S_{g,g}$ only has one nonzero term due to the
  assumptions on the parameter values $(a,b)$ and on the support of
  $g$. Indeed, in this case, $S_{g,g}$ becomes a multiplication
  operator. In Theorem~\ref{thm:dualFrameEqs} it is the Walnut
  representation of the \emph{mixed} frame operator $S_{g,h}$ that,
  due the support of $g$ and $h$, is guaranteed to have only one
  nonzero term.
\end{rem}

\begin{cor}\label{thm:BsplinesNEW}
Let $n \in \N, n \geq 2$, and let $(a,b) \in \Sigma_n$.  
Then $\mathcal{G}(B_n,a,b)$ is a frame for $L^2(\R)$. 
\end{cor}

\section{The counterexamples}
\label{sec:counter-example}

\subsection{Preliminaries}
We first need to set up notation and recall the Zibulski-Zeevi
representation. The Zak transform of a function $f \in L^2(\R)$ is defined as
\begin{equation}\label{eq:zakTransform}
\left(Z_{\lambda}f\right)(x,\nu)
= \sqrt{\lambda}\sum_{k\in\mathbb{Z}} f(\lambda(x-
  k))\myexp{2\pi i k \nu}, \quad a.e.\ x, \nu \in \mathbb{R},
\end{equation}
with convergence in  $L^2_\mathrm{loc}(\R)$. The Zak transform
$Z_\lambda$ is a unitary map of $L^2(\R)$ onto $L^2(\itvco{0}{1}^2)$, and
 it has the following quasi-periodicity:
\[
Z_\lambda f(x+1,\nu)= \myexp{2\pi i\nu} Z_\lambda f(x,
\nu),  \quad  Z_\lambda f(x, \nu +1) = Z_\lambda f(x, \nu) \quad \text{for 
a.e. } x,\nu \in \R.
\]  

We consider rationally oversampled Gabor systems, i.e., $\mathcal{G}(g,a,b)$ with
\[
ab \in \mathbb{Q}, \quad ab=\frac{p}{q} \quad \gcd(p,q)=1.
\]
For $g\in L^2(\R)$, we define column vectors $\phi^g_\ell(x,\nu) \in \C^p$ for $\ell
\in \set{0,1, \dots, q-1}$ by
\[ 
\phi^g_\ell(x,\nu) = \left(p^{-\frac{1}{2}} (Z_{\frac{1}{b}}g)(x-\ell
  \frac{p}{q},\nu+\frac{k}{p})\right)_{k=0}^{p-1} \ a.e. \ x,\nu \in \mathbb{R},
\] 
and column vectors $\psi^g_\ell(x,\nu) \in \C^p$ for $\ell
\in \set{0,1, \dots, q-1}$  by
\[ 
\psi^g_\ell(x,\nu) = \left(b^{-\tfrac12}\sum_{n\in\mathbb{Z}} g(x+aqn+a\ell+
   k/b) \myexp{-2\pi i aqn \nu}  \right)_{k=0}^{p-1} \ a.e. \ x,\nu \in \mathbb{R}.
\] 
The $p\times q$ matrix defined by $\Phi^g(x,\nu)=[\phi^g_\ell(x,\nu)]_{\ell=0}^{q-1}$ is
the so-called Zibulski-Zeevi matrix, while the $p\times q$ matrix
$\Psi^g(x,\nu)=[\psi^g_\ell(x,\nu)]_{\ell=0}^{q-1}$ is closely related to the
Zibulski-Zeevi matrix and appears in
\cite{MR3027914,JakobsenCocompact2014}. Theorem~\ref{thm:ZZ_singular_values} below says
that the Gabor system $\mathcal{G}(g,a,b)$ is a frame with (optimal) bounds $A$ and $B$
for $L^2(\R)$ if and only if $\sqrt{A}$ is the infimum over a.e.\ $(x,\nu)\in \R^2$ of the
smallest singular value of $\Phi^g(x,\nu)$ (or $\Psi^g(x,\nu)$) and $\sqrt{B}$ is the
supremum over a.e.\ $(x,\nu)\in \R^2$ of the largest singular value of $\Phi^g(x,\nu)$ (or
$\Psi^g(x,\nu)$).
\begin{thm}
\label{thm:ZZ_singular_values}
  Let $A,B\ge 0$, and let $g \in L^2(\R)$. Suppose $\mathcal{G}(g,a,b)$
  is a rationally oversampled Gabor system. Then the following
  assertions are equivalent:
\begin{enumerate}[(i)]
\item   $\mathcal{G}(g,a,b)$ is a Gabor frame for $L^2(\R)$ with bounds $A$ and $B$,
\item $\set{\phi^g_\ell(x,\nu)}_{\ell=0}^q$ is a frame for $\C^p$ with
  uniform bounds $A$ and $B$ for a.e. $(x,\nu) \in
  \itvcos{0}{1}^2$. 
\item $\set{\psi^g_\ell(x,\nu)}_{\ell=0}^q$ is a frame for $\C^p$ with
  uniform bounds $A$ and $B$ for a.e. $(x,\nu) \in
  \itvcos{0}{a} \times \itvcos{0}{\frac{1}{aq}}$. 
\end{enumerate}
\end{thm}

\subsection{A partition of unity property and zeros of the Zak transform}
\label{sec:partition-unity-zero}

We let $N_n$ denote the $n$th order cardinal B-spline, i.e.,
$N_n(x)=B_n(x-\frac{n}{2})$ for each $n\in \N$.  Since the frame set
$\cF(g)$ is invariant under translation of the generator $g$, it
follows that $\cF(N_n)=\cF(B_n)$. The results in this section are
formulated for $N_n$; the results can, of course, also be formulated
for $B_n$. 

Integer translations of the B-splines form a partition of unity. The
following result shows that the partition of unity property for
B-splines is somewhat stable under perturbations of the translation
lattice; indeed, the translations of B-splines along $c^{-1}\Z$ yield
a \emph{partly} partition of unity (up to a constant) whenever $c$ is
sufficiently close to an integer. For $x\in \R$ let $\round{x}$ denote
the round function to the nearest integer, i.e.,
$R(x)=\floor{x+\frac12}$, and let $\sfrac{x}=x-\round{x} \in
\itvoc{-\frac12}{\frac12}$ denote the (signed) fractional part of $x$.
\begin{lem}
\label{lem:partly-pou}
  Let $n \in \N$ and $c>0$. Assume that $\abs{\sfrac{c}} \le
  \frac{1}{n}$.
  \begin{enumerate}[(i)]
  \item If $\sfrac{c}\ge 0$, then 
    \begin{equation}
      \label{eq:partly-part-of-unity-1}
      \sum_{k \in \Z} N_n((x+k)/c) = \mathrm{const} \quad \text{for }
      x \in \bigcup_{m\in \Z}\itvcc{m+n\sfrac{c}}{m+1} 
    \end{equation}
  \item If $\sfrac{c}\le 0$, then 
    \begin{equation}
      \label{eq:partly-part-of-unity-2}
      \sum_{k \in \Z} N_n((x+k)/c) = \mathrm{const} \quad \text{for }
      x \in \bigcup_{m\in \Z}\itvcc{m}{m+1+n\sfrac{c}} 
    \end{equation}
  \end{enumerate}
\end{lem}
\begin{proof}
  (i): Assume that $\sfrac{c}\ge 0$. The proof
  goes by induction. For $n=1$ it is immediate that
  \eqref{eq:partly-part-of-unity-1} holds. Suppose that
  \eqref{eq:partly-part-of-unity-1} holds for $n \in \N$. Hence, there
  exist a 1-periodic function $g \in L^\infty(\R)$ and a constant
  $d \in \R$ such that
    \begin{equation*}
      \sum_{k \in \Z} N_n((x+k)/c) = d \chi_A(x) + g(x)\chi_{A^c}(x),
      \quad \text{for } x\in \R, 
    \end{equation*}
where $A:=\bigcup_{m\in \Z}\itvcc{m+n\sfrac{c}}{m+1}$. 
We convolve the left and right hand side of this equation with
$N_1(\tfrac{\cdot}{c})$. The left hand side then becomes $\sum_{k \in
  \Z} N_{n+1}((x+k)/c)$. The right hand side becomes
\[
  (N_1(\tfrac{\cdot}{c}) \ast (d\, \chi_A + g \cdot \chi_{A^c}))(x) =
  \int_{-c+x}^x \left(d \,
  \chi_A(t) + g(t)\, \chi_{A^c}(t) \right) \, dt,
\]
For $x \in B:=\bigcup_{m\in \Z}\itvcc{m+(n+1)\sfrac{c}}{m+1}$, 
we have $x \in A$ and $-c+x \in A$. Thus $N_1(\tfrac{\cdot}{c}) \ast
(c \chi_A + g \cdot \chi_{A^c})$ is constant on $B$. This completes
the proof of the inductive step, and (i) is proved. The proof of (ii) is similar.
\end{proof}

\begin{rem}
\label{rem:zak-Bn-continuous}
  The Zak transform of B-splines $Z_\lambda N_n$ is defined pointwise
  and is bounded on $\R^2$. For $n \ge 2$, the Zak transform
  $Z_\lambda N_n$ is continuous, while $Z_\lambda N_1$ is piecewise
  continuous on $\itvco{0}{1}^2$ with 
  discontinuities along at most finitely many lines of the form $t=\mathrm{const}$.
\end{rem}
The next result shows that, for $b>\frac32$, the Zak transform
$Z_{1/b}N_n(x,\nu)$ of the B-spline $N_n$  for $(x,\nu)\in
\itvoc{0}{1}^2$ is zero along $\round{b}-1$
parallel lines of length $1-n\abs{\sfrac{b}}$.
\begin{prop}
\label{thm:zeros-zak}
    Let $n \in \N$ and $b>\frac32$. Assume that $\abs{\sfrac{b}} \le
    \frac1n$.
 \begin{enumerate}[(i)]
  \item If $\sfrac{b}\ge 0$, then 
    \begin{equation}
Z_{1/b}N_n(x,\tfrac{k}{\round{b}}) = 0 \quad \text{for } x \in
\bigcup_{m\in \Z}\itvcc{m+n\sfrac{b}}{m+1}, k \in \Z \setminus \round{b}\Z.
\label{eq:zeros-zak-1}
    \end{equation}
  \item If $\sfrac{b}\le 0$, then 
    \begin{equation}
Z_{1/b}N_n(x,\tfrac{k}{\round{b}}) = 0 \quad \text{for }
      x \in \bigcup_{m\in \Z}\itvcc{m}{m+1+n\sfrac{b}}, k \in \Z
      \setminus \round{b}\Z. 
\label{eq:zeros-zak-2}
    \end{equation}
  \end{enumerate}
\end{prop}
\begin{proof}
We will only prove (i) as the proof of (ii) is similar. So assume that
$\sfrac{b}\ge 0$. Let $\nu_0 = s/\round{b}$ for
some $s \in \Z \setminus \round{b}\Z$.
We rewrite the Zak transform of $N_n$ as follows:
\begin{align}
  Z_{1/b}N_n(x,\nu_0) &= \sum_{k \in \Z} N_n(\tfrac1b(x+k)) \myexp{-2\pi
    i k \tfrac{s}{\round{b}}} \nonumber \\
 &= \sum_{\ell=0}^{\round{b}-1} \sum_{r\in \Z}
 N_n(\tfrac1b(x+\round{b}r+\ell)) \myexp{-2\pi
    i \ell \tfrac{s}{\round{b}}}. \label{eq:zak-rewrite}
\end{align}
We claim that there exists a constant $d \in \R$ such that for every
$\ell\in \set{0,1,\dots,\round{b}-1}$:
\[
\sum_{r\in \Z}
 N_n(\tfrac1b(x+\round{b}r+\ell)) = d \qquad \text{for }  x \in
\bigcup_{m\in \Z}\itvcc{m+n\sfrac{b}}{m+1}.
\]
Fix $\ell \in \set{0,1,\dots,\round{b}-1}$ for a moment. Since 
\[
\frac1b(x+\round{b}r+\ell)=\frac{\round{b}}{b}\bigl((x+\ell)/\round{b}+r\bigr),
\]
it follows from Lemma~\ref{lem:partly-pou} that $\sum_{r\in \Z}
 N_n(\tfrac1b(x+\round{b}r+\ell))$ is constant for 
\[ 
  \frac{x+\ell}{\round{b}} \in \bigcup_{m\in \Z}\itvcc{m+n\sfrac{\tfrac{b}{\round{b}}}}{m+1},
\]
 that is, for
\[ 
  x \in \left(\bigcup_{m\in \Z}\itvcc{\round{b}m+n\sfrac{b}}{\round{b}(m+1)}\right)-\ell.
\]
Taking the intersection over $\ell \in \set{0,1,\dots,\round{b}-1}$
completes the proof of the claim. 

Hence, for $x \in \bigcup_{m\in \Z}\itvcc{m+n\sfrac{c}}{m+1}$, we can
continue \eqref{eq:zak-rewrite}: 
\begin{equation*}
  Z_{1/b}N_n(x,\nu_0) = \sum_{\ell=0}^{\round{b}-1} d \myexp{-2\pi
    i \ell \tfrac{s}{\round{b}}} = 0.
\end{equation*}
\end{proof}
Note that method of the above proof relies crucially  on the
assumption $b>\frac32$ which guarantees $\round{b}\ge 2$.

\subsection{A family of counterexamples}

The main result of this note is the following family of obstructions
for the frame property of $\gaborG{B_n}$.
\label{sec:proof-main-theorem}
\begin{thm}
\label{thm:frame-set-Bn}
  Let $n \in\N$, let $a>0,b > \frac32$, and let $p,q \in \N$, where
  $\gcd{(p,q)}=1$. If
 \[ 
 ab=\frac{p}{q} \quad \text{and} \quad \abs{\sfrac{b}} \le \frac{1}{nq},
\]
then $\cG (B_n,a,b)$ is \emph{not}
  a frame for $L^2(\R)$. 
\end{thm}

\begin{proof}
  Let $n \in \N$ and let $ab=\frac{p}{q}$, where $b > \frac32$ and
  $\abs{\sfrac{b}} \le \frac{1}{nq}$.  Assume that $\sfrac{b} \ge
  0$. Take $x_0=1/q$ and $\nu_0=1/\round{b}$. Using the
  quasi-periodicity of the Zak transform, the first coordinate of the
  vectors $\set{\phi^{N_n}_\ell(x_0,\nu_0)}_{\ell=0}^q$ can be written
  (up to a scalar multiplication) as:
  \begin{equation}
    Z_{1/b}N_n(\tfrac{\ell}{q},\nu_0) \qquad \ell =1,2,\dots, q.\label{eq:row-1}
  \end{equation}
  It follows from Proposition~\ref{thm:zeros-zak} that
  $Z_{1/b}N_n(x,\nu_0)$ is zero for
  $x\in \itvcc{n\sfrac{b}}{1}$. Since $n\sfrac{b}\le \tfrac1q$, it
  follows that the first coordinate of all the vectors in
  $\set{\phi^{N_n}_\ell(x_0,\nu_0)}_{\ell=0}^q$ is zero, hence
  $\set{\phi^{N_n}_\ell(x_0,\nu_0)}_{\ell=0}^q$ is not a frame. By
  Theorem~\ref{thm:ZZ_singular_values} and
  Remark~\ref{rem:zak-Bn-continuous}, this shows that $\gaborG{N_n}$
  is not a frame for $L^2(\R)$. The proof of the case $\sfrac{b} <
  0$ is similar. 
\end{proof}

\begin{rem}
  \label{rem:non-optimal-range}
  The range $\abs{\sfrac{b}} \le \frac{1}{nq}$ in
  Theorem~\ref{thm:frame-set-Bn} might not be optimal in the sense
  that the non-frame property of $\gaborG{N_n}$ along the hyperbola
  $ab=\frac{p}{q}$ might extend beyond the range $\abs{\sfrac{b}} \le
  \frac{1}{nq}$. Indeed, in the proof of
  Theorem~\ref{thm:frame-set-Bn}, we show that the Zibulski-Zeevi
  matrix $\Phi^{B_n}(x,\nu)$ contains a row of zeros for some choice
  of $(x,\nu)$, but to have a non-frame property of $\gaborG{N_n}$,
  one just needs that $\Phi^{B_n}(x,\nu)$ is not of full
  rank. However, for $p=1$, that is, when $\Phi^{B_n}(x,\nu)$ is a row
  vector, we believe the range is optimal. This believe is supported
  by the fact the $\abs{\sfrac{b}} \le \frac{1}{q}$ is optimal for the
  first B-spline $N_1$, see \cite[Theorem 2.2]{DaiABC2015}.
\end{rem}

Note that any sufficiently nice function $g$ that generates a partition of unity and that
yields a partly partition of unity under perturbation of the integer
lattice as in Lemma~\ref{lem:partly-pou} for $c \approx 1$ will have
the a family of obstructions for the frame property of $\gaborG{g}$
as in Theorem~\ref{thm:frame-set-Bn}. 

\subsection{A different type of counterexamples}
\label{sec:other-type-obstr}

The following result shows that the exist other counterexamples, not
covered by those presented in Theorem~\ref{thm:frame-set-Bn}. The
hyperbolic curves from Theorem~\ref{thm:frame-set-Bn} together with
the horizontal lines $b=2,3,\dots$ form path-connected sets around
$b=2,3,\dots$ where $\gaborG{B_n}$ is not a frame. In the following
counterexamples we have $b\in\itvcc{\frac73}{\frac83}$ which is ``far'' from the two
path-connected sets at $b=2$ and $b=3$, see
Figure~\ref{fig:frame-set-B2}.

\begin{thm}\label{thm:NEWnonFrame}
If $ab=\frac56$ and $b\in\itvcc{\frac73}{\frac83}$, then
the Gabor system $\mathcal{G}(B_2,a,b)$ is not a
frame for $L^2(\R)$.
\end{thm}

\begin{proof}
Since $ab=\frac{5}{6}$, the Zibulski-Zeevi type matrix $\Psi^{B_2}(x,\nu)$
for the Gabor system $\mathcal{G}(B_2,a,b)$ is of
size $5 \times 6$. At $(x,\nu)=(0,0)$ it reads:
\[
\Psi^{B_2}(0,0)=\left[ b^{-\tfrac12}\sum_{n\in\mathbb{Z}} B_2(6an+a\ell+k/b) 
\right]_{\substack{k=0,\dots,4\\\ell=0,\dots,5}}. 
% = \frac{\sqrt{10}}{75}
% \begin{bmatrix}
%  15&10&5&0&5&10 \\ 9&4&1&6&11&14 \\ 3&2&7&12&13&8 \\ 3&8&13&12&7&2 \\ 9&14&11&6&1&4
% \end{bmatrix}.
% \]
% The reduced row echelon form of $\Psi^{B_2}(0,0)$ is:
% \[\Psi^{B_2}(0,0) \xlongrightarrow{\text{Gauss-Jordan elimination}} \begin{bmatrix}
% 1 & 0 & 0 & 0 &  1 &  1 \\
% 0 & 1 & 0 & 0 & -1 &  0 \\
% 0 & 0 & 1 & 0 &  0 & -1 \\
% 0 & 0 & 0 & 1 &  1 &  1 \\
% 0 & 0 & 0 & 0 &  0 &  0
% \end{bmatrix}.
\]
The $(a,b)$-values in question can be expressed as $1/b=6a/5$ for $a
\in \itvcc{\frac5{16}}{\frac5{14}}$. We consider the entries in
$\Psi^{B_2}(0,0)$ as $6a$-periodizations of $B_2$ at sampling locations
$a\ell+k/b$, that is, at $a\ell+\tfrac65 ak \pmod{6a}$. The $2$nd  and $5$th row of
$\Psi^{B_2}(0,0)$ are sampled at locations 
\begin{align*} k=1: \quad &\frac{6}{5}a,\frac{11}{5}a,-\frac{14}{5}a,-\frac{9}{5}a,-\frac{4}{5}a,\frac{1}{5}a,\mod{6a}
\intertext{and}
k=4: \quad  &-\frac{6}{5}a,-\frac{1}{5}a,\frac{4}{5}a,\frac{9}{5}a,\frac{14}{5}a,-\frac{11}{5}a, \mod{6a},
\end{align*}
respectively. Note that the $6a$-periodizations of $B_2$ at any sampling location
in the interval $\itvcc{-\tfrac{14}{5}a}{\tfrac{14}{5}a} \bmod{6a}$
only has one nonzero term. Therefore, by definition of $B_2$, we directly
see that
\[ 
R_2-R_5 =
\begin{bmatrix}
  0 & -2a & -2a & 0 & 2a & 2a
\end{bmatrix}
\]
where $R_i$ denotes the $i$th row of $\Psi^{B_2}(0,0)$. In the same
way, we see that 
\[ 
R_3-R_4 =
\begin{bmatrix}
  0 & -a & -a & 0 & a & a
\end{bmatrix}
\]
for  $a
\in \itvcc{\frac5{16}}{\frac5{14}}$.
We see that $\Psi^{B_2}(0,0)$ does not have full rank, implying that the
smallest singular values is zero. By
Theorem~\ref{thm:ZZ_singular_values}, the lower frame bound of
$\mathcal{G}(B_2,a,\frac{5}{6a})$ is zero for $a
\in \itvcc{\frac5{16}}{\frac5{14}}$. 
\end{proof}

Theorem~\ref{thm:NEWnonFrame} raises a number of questions. A first
question is
whether the non-frame property along $ab=\tfrac56$ in
extends beyond
$b\in\itvcc{\frac73}{\frac83}$. The numerical computations in Figure~\ref{fig:plots}
suggest that $b\in \itvcc{\frac73}{\frac83}$ in
Theorem~\ref{thm:NEWnonFrame} is, at the least, very close to being
optimal. We remark that the method of proof for
Theorem~\ref{thm:NEWnonFrame} breaks down for
$b\notin\itvcc{\frac73}{\frac83}$ because these 
$(a,b)$-values no longer guarantee sampling locations
in $\itvcc{-\tfrac{14}{5}a}{\tfrac{14}{5}a} \bmod{6a}$, which 
causes some of the $6a$-periodizations of $B_2$ to have more than one nonzero
term. 
\begin{figure}[!h]
\begin{minipage}[t]{.45\textwidth}
\centering
\definecolor{mycolor1}{rgb}{0.00000,0.44700,0.74100}%
\begin{tikzpicture}[scale=0.5]
\begin{axis}[%
width=4.602in,
height=3.513in,
at={(0.772in,0.553in)},
scale only axis,
xmin=0,
xmax=0.4,
xlabel={$a$},
ymin=0,
ymax=0.5,
anchor=west,
ylabel={$\sqrt{A}$},
yticklabels={0,0,0.05,0.10,0.15,0.20,0.25,0.30,0.35,0.40,0.45,0.50},
xticklabels={0,0,0.05,0.10,0.15,0.20,0.25,0.30,0.35,0.40},
axis background/.style={fill=white}
]
\addplot [thick,color=mycolor1,solid,forget plot]
  table[row sep=crcr]{%
0.016	0.481554188685709\\
0.0166666666666667	0.471434296654433\\
0.0173913043478261	0.461843809623922\\
0.0181818181818182	0.451270561386478\\
0.019047619047619	0.441250117097757\\
0.02	0.43009378586251\\
0.0210526315789474	0.419637983443951\\
0.0222222222222222	0.407855598963687\\
0.0235294117647059	0.396842235746743\\
0.025	0.384253474574125\\
0.0266666666666667	0.372642154425484\\
0.0285714285714286	0.359090557923662\\
0.0307692307692308	0.346730806715969\\
0.032	0.340509471051936\\
0.0333333333333333	0.331964154428668\\
0.0347826086956522	0.326571762044522\\
0.0363636363636364	0.318726945547504\\
0.0380952380952381	0.312009576264557\\
0.04	0.302287394752538\\
0.0421052631578947	0.296726638502169\\
0.0444444444444444	0.287852404237674\\
0.0470588235294118	0.280606966047634\\
0.048	0.27801589882223\\
0.05	0.269136651620554\\
0.0521739130434783	0.266630786398686\\
0.0533333333333333	0.263492673356918\\
0.0545454545454545	0.260529689512641\\
0.0571428571428571	0.253138764667445\\
0.06	0.248301610134073\\
0.0615384615384615	0.245168524714611\\
0.0631578947368421	0.242253740657526\\
0.064	0.240409609784497\\
0.0666666666666667	0.230760056122787\\
0.0695652173913043	0.230566653015877\\
0.0705882352941176	0.229072966849733\\
0.0727272727272727	0.225358563635028\\
0.075	0.221813843201186\\
0.0761904761904762	0.219697271299221\\
0.08	0.212455977107009\\
0.0833333333333333	0.21019115836236\\
0.0842105263157895	0.208972201320256\\
0.0857142857142857	0.207276661140634\\
0.0869565217391304	0.20617399047615\\
0.0888888888888889	0.20351769888028\\
0.0909090909090909	0.20144059011252\\
0.0923076923076923	0.200097098247397\\
0.0941176470588235	0.198107126771845\\
0.0952380952380952	0.196968739445464\\
0.096	0.196331460756506\\
0.1	0.182854068567043\\
0.104347826086957	0.188133706858536\\
0.105263157894737	0.187266988633506\\
0.106666666666667	0.185999082476694\\
0.109090909090909	0.18378798377336\\
0.111111111111111	0.18194686485373\\
0.112	0.180991717003191\\
0.114285714285714	0.178914314778492\\
0.116666666666667	0.177684152002844\\
0.117647058823529	0.177055681593351\\
0.121739130434783	0.164734120208954\\
0.123076923076923	0.169774542245311\\
0.125	0.16939120767278\\
0.126315789473684	0.169224028127652\\
0.127272727272727	0.168131510332278\\
0.128	0.168882621068234\\
0.133333333333333	0.160343667405904\\
0.139130434782609	0.161855428459018\\
0.14	0.160738212851041\\
0.141176470588235	0.160260909834545\\
0.142857142857143	0.158660398870734\\
0.144	0.158474859226892\\
0.145454545454545	0.156360244096958\\
0.147368421052632	0.149500178532873\\
0.15	0.155093779942093\\
0.152380952380952	0.155124315788888\\
0.153846153846154	0.154409992789569\\
0.155555555555556	0.153277623368531\\
0.156521739130435	0.153620921241946\\
0.16	0.150053775613128\\
0.163636363636364	0.15001927885823\\
0.164705882352941	0.148123803597367\\
0.166666666666667	0.147801852978869\\
0.168421052631579	0.14728720842158\\
0.171428571428571	0.145462319159406\\
0.173913043478261	0.140553063503745\\
0.175	0.139976635652744\\
0.176	0.143125306023907\\
0.177777777777778	0.143493215374072\\
0.18	0.142853841926158\\
0.181818181818182	0.141699193999658\\
0.183333333333333	0.140530921394936\\
0.184615384615385	0.139248669735021\\
0.186666666666667	0.134830219366599\\
0.188235294117647	0.135054650334924\\
0.189473684210526	0.134255848265663\\
0.19047619047619	0.133280740347127\\
0.191304347826087	0.132964394055285\\
0.192	0.132161587895374\\
0.2	0.104588552144436\\
0.208	0.123215167892562\\
0.208695652173913	0.123116649388812\\
0.20952380952381	0.123196864581109\\
0.210526315789474	0.122670042062714\\
0.211764705882353	0.12267309933969\\
0.213333333333333	0.122184408392654\\
0.215384615384615	0.119219479627642\\
0.216666666666667	0.12204082949289\\
0.218181818181818	0.124645859745301\\
0.22	0.125644265870533\\
0.222222222222222	0.126367349647862\\
0.224	0.126680613008359\\
0.225	0.126461119349195\\
0.226086956521739	0.126267992580142\\
0.228571428571429	0.124859928112541\\
0.231578947368421	0.120362050960607\\
0.233333333333333	0.114801193846278\\
0.235294117647059	0.112168003063303\\
0.236363636363636	0.112714240321407\\
0.24	0.118300506920811\\
0.243478260869565	0.113401558350849\\
0.244444444444444	0.114660967646335\\
0.246153846153846	0.117037340935428\\
0.247619047619048	0.117304219548924\\
0.25	0.116634952755509\\
0.25	0.116634952755509\\
0.252631578947368	0.116727824879107\\
0.254545454545455	0.115191690318571\\
0.256	0.116729102613846\\
0.257142857142857	0.116706015256214\\
0.258823529411765	0.116732262961595\\
0.26	0.116222028616305\\
0.260869565217391	0.116157955819874\\
0.266666666666667	0.111325939668673\\
0.272	0.113479842641764\\
0.272727272727273	0.113293220690652\\
0.273684210526316	0.113151936262176\\
0.275	0.112906760990954\\
0.276923076923077	0.112555981237366\\
0.278260869565217	0.112056614299611\\
0.28	0.110665371243933\\
0.282352941176471	0.111128756841574\\
0.283333333333333	0.110882184208691\\
0.285714285714286	0.110097382374463\\
0.285714285714286	0.110097382374463\\
0.288	0.110095740746634\\
0.288888888888889	0.109546888553266\\
0.290909090909091	0.108811899951617\\
0.293333333333333	0.106989807827584\\
0.294736842105263	0.10481570070677\\
0.295652173913043	0.104373684165459\\
0.3	0.0987331499563241\\
0.304	0.0731043599521697\\
0.304761904761905	0.0684847010376727\\
0.305882352941176	0.0620484301237305\\
0.307692307692308	0.0525132431307488\\
0.309090909090909	0.0458356849499166\\
0.311111111111111	0.0372046796383018\\
0.31304347826087	0.0306809159781879\\
0.314285714285714	0.0259233372887492\\
0.315789473684211	0.0215497300141748\\
0.316666666666667	0.0190862468290263\\
0.32	0.0103749163559989\\
0.323809523809524	0.00548342305409309\\
0.325	0.00411065596243283\\
0.327272727272727	0.0034780326599614\\
0.329411764705882	0.000860165478839569\\
0.330434782608696	0.000463541324180909\\
0.333333333333333	1.54775264602868e-17\\
0.336	0.000379947716201682\\
0.336842105263158	0.000601519914956391\\
0.338461538461538	0.00215326457466025\\
0.34	0.00198194665638185\\
0.342857142857143	0.00375490558469271\\
0.345454545454545	0.00568871657606282\\
0.346666666666667	0.00666155517539556\\
0.347826086956522	0.00761389847341775\\
0.35	0.00939731077776783\\
0.352	0.0112280285279755\\
0.352941176470588	0.0117200170005918\\
0.355555555555556	0.00392143754110717\\
0.357894736842105	0.000950252829385864\\
0.36	0.00327350770828235\\
0.361904761904762	0.00138137883401489\\
0.363636363636364	0.000675089002007438\\
0.365217391304348	0.000965808791899878\\
0.366666666666667	0.0022079344085758\\
0.368	0.00482892174306198\\
0.369230769230769	0.00678044827417464\\
0.371428571428571	0.00850652778679055\\
0.373333333333333	0.00853007314577968\\
0.375	0.00866377563027741\\
0.376470588235294	0.00871079842021392\\
0.377777777777778	0.00866264705038671\\
0.378947368421053	0.00778526697094949\\
0.38	0.00675262023279902\\
0.380952380952381	0.00593230501865234\\
0.381818181818182	0.00519657379840492\\
0.382608695652174	0.00440382583080068\\
0.383333333333333	0.00359615734867351\\
0.384	0.0029200361148488\\
};
 \addplot [thick,color=red,solid,forget plot]
   table[row sep=crcr]{%
 0.333333333333333	0.00\\
 0.333333333333333	0.02\\
 };
\end{axis}
\end{tikzpicture}
\subcaption{Estimate of $\sqrt{A}$ along the horizontal line
          $(a,\tfrac52)$ for $a\in
          \itvcc{0}{\tfrac25}$ in the $(a,b)$-plane.}
       % \label{fig:b25}
    \end{minipage}
\hspace{1em}
    \begin{minipage}[t]{0.45\textwidth}
        \centering
\definecolor{mycolor1}{rgb}{0.00000,0.44700,0.74100}%
%
%\pgfplotsset{scaled y ticks=false}
\begin{tikzpicture}[scale=0.5]
\begin{axis}[%
width=4.521in,
height=3.559in,
at={(0.758in,0.488in)},
scale only axis,
xmin=2.2,
xmax=2.8,
xlabel={$b$},
ymin=0,
ymax=0.025,
ylabel={$\sqrt{A}$},
axis background/.style={fill=white},
yticklabel style={
        /pgf/number format/fixed,
        /pgf/number format/precision=3
},
scaled y ticks=false
]
\addplot [thick,color=mycolor1,solid,forget plot]
  table[row sep=crcr]{%
2.15	0.0114969470399483\\
2.16	0.0129964197662096\\
2.17	0.0145993357883583\\
2.18	0.0166062417795031\\
2.19	0.0185706134320666\\
2.2	0.0204822730110974\\
2.21	0.0197423202628218\\
2.22	0.0171158188691723\\
2.23	0.0143071573599921\\
2.24	0.0113216079517324\\
2.25	0.00850713987729348\\
2.26	0.00763398589705999\\
2.27	0.00703409024346112\\
2.28	0.00633050345632263\\
2.29	0.00549385089700934\\
2.3	0.0045084066498478\\
2.31	0.00336867983058096\\
2.32	0.00202051305267011\\
2.33	0.000520845497747336\\
2.34	6.23951497978352e-17\\
2.35	6.34505493879639e-17\\
2.36	8.42179086681672e-17\\
2.37	4.23867131562466e-17\\
2.38	7.4033628929406e-17\\
2.39	4.6975131340903e-17\\
2.4	3.48149729846049e-17\\
2.41	5.99946355109227e-17\\
2.42	2.8699961221436e-17\\
2.43	2.75270035393601e-17\\
2.44	2.04337210164666e-17\\
2.45	1.64648135693033e-17\\
2.46	1.98639995028736e-17\\
2.47	1.18260623880681e-17\\
2.48	1.04487577702665e-17\\
2.49	2.71163397515187e-17\\
2.5	1.54775264602868e-17\\
2.51	2.25216950969598e-17\\
2.52	2.22567414090237e-17\\
2.53	1.54241310094722e-17\\
2.54	2.58599713310639e-17\\
2.55	2.9212219358232e-17\\
2.56	1.67434227683358e-17\\
2.57	2.97210593354904e-17\\
2.58	2.85218906677723e-17\\
2.59	2.2839038507597e-17\\
2.6	2.7980112913949e-17\\
2.61	3.80062139739359e-17\\
2.62	3.57566843218655e-17\\
2.63	5.08363117884303e-17\\
2.64	5.98442346875835e-17\\
2.65	3.5626367500237e-17\\
2.66	3.89792478174825e-17\\
2.67	0.000454520602903055\\
2.68	0.00174939014457642\\
2.69	0.00289491568037532\\
2.7	0.0038448721606382\\
2.71	0.0046498756775937\\
2.72	0.00531772221646339\\
2.73	0.00586441213562912\\
2.74	0.00631659348719815\\
2.75	0.00635122524369842\\
2.76	0.00631494446551022\\
2.77	0.00621598172788005\\
2.78	0.00599042927154455\\
2.79	0.00565940556862505\\
2.8	0.00507373216579751\\
2.81	0.00435871299766403\\
2.82	0.00353654156900303\\
2.83	0.00266076989146642\\
2.84	0.00222991023751109\\
2.85	0.00229453615065115\\
};
\addplot [thick,color=red,solid,forget plot]
  table[row sep=crcr]{%
2.5	0\\
2.5	0.001\\
};
\end{axis}
\end{tikzpicture}%
        \subcaption{Estimate of $\sqrt{A}$ along the hyperbola
          $(\tfrac{5}{6b},b)$ for $b\in
          \itvcc{2.2}{2.8}$ in the $(a,b)$-plane.}
        %\label{fig:ab56}
      \end{minipage}
\caption{Numerical approximations (using Matlab) of the lower frame bound for
          $\mathcal{G}(B_2,a,b)$ near the point
          $(a,b)=(\tfrac13,\tfrac52)$ (marked with red).}
\label{fig:plots}
  \end{figure}
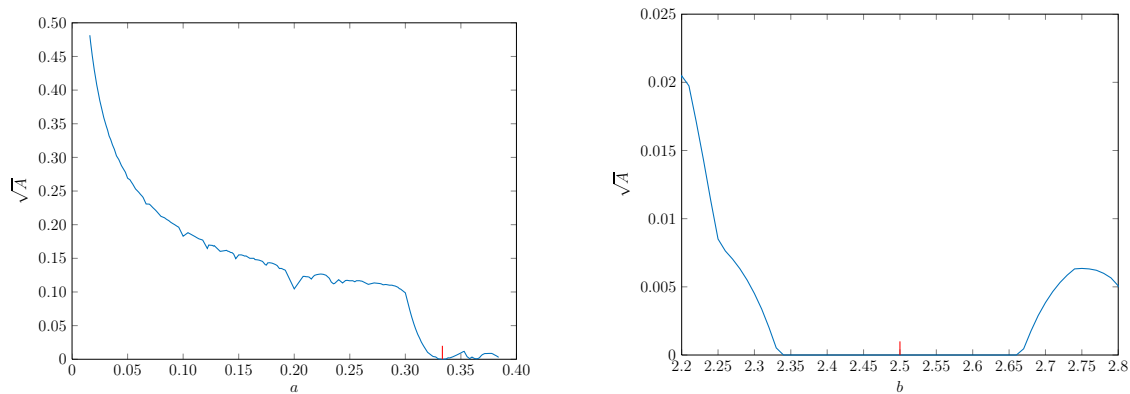
 A second question is whether this type  of counterexamples is
singular.  We do not believe that this is the case. 
In the spirit of \cite{MR3232589}, although not as
bold, we make a new conjecture. Our conjecture is based on
Theorem~\ref{thm:NEWnonFrame} and
exact symbolic calculations in Maple of the Zibulski-Zeevi representation
associated with  $\mathcal{G}(B_2,a,b)$.
\begin{conj}\label{con:NEW1}
The Gabor system $\mathcal{G}(B_2,a_0,b_0)$ is not a frame for 
\begin{equation}\label{eq:con_newNonFrame}
a_0= \frac{1}{2m+1}, \ b_0=\frac{2k+1}{2}, \ k,m\in\mathbb{N}, \ k>m, \ a_0 b_0 < 1.
\end{equation}
Furthermore, the Gabor system $\mathcal{G}(B_2,a,b)$ is not a frame along the hyperbolas
\begin{equation}\label{eq:con_hyperbelStykke}
\ ab=\frac{2k+1}{2\left(2m+1\right)}, \quad \text{for } b\in\left[b_0-a_0\frac{k-m}{2},b_0+a_0\frac{k-m}{2}\right],
\end{equation}
for every $a_0$ and $b_0$ defined by (\ref{eq:con_newNonFrame}).
\end{conj}
The conjecture is verified for the case $m=1$ and $k=2$ by Theorem~\ref{thm:NEWnonFrame}.

% We have verified the conjecture, including the ``furthermore''-part, in
% Maple using exact symbolic calculations for low values of $m$ and $k$, e.g.,
% for the case $m=1$ and $k=2$. 
% Symbolic calculations in Maple also suggest that
% $(a_0,b_0)=(\frac{1}{2m},\frac{2k+1}{2})\notin \cF(B_2)$ for $k>m$.

 \smallskip
 \paragraph{Acknowledgments.}

 The authors would like to thank Ole Christensen for posing the
 problem of characterizing the frame set of B-splines in a talk at the
 Technical University of Denmark on February 3, 2015, that initiated
 the work presented in this paper. The authors would also like to
 thank Karlheinz Gr\"ochenig for comments improving the presentation
 of the paper.

\def\cprime{$'$} \def\cprime{$'$} \def\cprime{$'$}
  \def\uarc#1{\ifmmode{\lineskiplimit=0pt\oalign{$#1$\crcr
  \hidewidth\setbox0=\hbox{\lower1ex\hbox{{\rm\char"15}}}\dp0=0pt
  \box0\hidewidth}}\else{\lineskiplimit=0pt\oalign{#1\crcr
  \hidewidth\setbox0=\hbox{\lower1ex\hbox{{\rm\char"15}}}\dp0=0pt
  \box0\hidewidth}}\relax\fi} \def\cprime{$'$} \def\cprime{$'$}
  \def\cprime{$'$} \def\cprime{$'$} \def\cprime{$'$} \def\cprime{$'$}

 % \bibliographystyle{abbrv} % eller unsrt, abbrv eller alpha
 % \bibliography{mitnyebib} % bemaerk ingen endelse

 \end{document}